\documentclass[12pt,a4paper]{article}
\usepackage[utf8]{inputenc}
\usepackage{amsmath}
\usepackage{amsfonts}
\usepackage{amssymb}
\usepackage{amsthm}
\usepackage{hyperref}
\usepackage[left=1.9cm,right=1.9cm,top=2cm,bottom=2cm]{geometry}
\usepackage{xcolor}
\usepackage[all,cmtip]{xy}
\usepackage{tikz}
\usetikzlibrary{arrows}
\usetikzlibrary{decorations.markings}
\usepackage{enumerate}

\newcommand{\id}{\mathrm{id}}

\newcommand{\rk}{\mathrm{rk}}
\newcommand{\codim}{\mathrm{codim}\,}

\newtheorem{thm}{Theorem}[section]
\newtheorem*{thm*}{Theorem}
\newtheorem*{cor*}{Corollary}
\newtheorem{prop}[thm]{Proposition}
\newtheorem*{prop*}{Proposition}
\newtheorem{cor}[thm]{Corollary}
\newtheorem{lem}[thm]{Lemma}

\theoremstyle{definition}
\newtheorem{rem}[thm]{Remark}
\newtheorem{defn}[thm]{Definition}
\newtheorem{ex}[thm]{Example}

\begin{document}

\title{Reconstructing the orbit type stratification of a torus action from its equivariant cohomology}
\author{Oliver Goertsches\footnote{Philipps-Universit\"at Marburg, email:
goertsch@mathematik.uni-marburg.de}\,\, and Leopold
Zoller\footnote{Max Planck Institute for Mathematics in Bonn, email: zoller@math.lmu.de}}

\maketitle

\begin{abstract} 
We investigate what information on the orbit type stratification of a torus action on a compact space is contained in its rational equivariant cohomology algebra. Regarding the (labelled) poset structure of the stratification we show that equivariant cohomology encodes the subposet of ramified elements. For equivariantly formal actions, we also examine what cohomological information of the stratification is encoded. In the smooth setting we show that under certain conditions -- which in particular hold for a compact orientable manifold with discrete fixed point set -- the equivariant cohomologies of the strata are encoded in the equivariant cohomology of the manifold.
\end{abstract}

\section{Introduction}
Motivated by Masuda's equivariant cohomological rigidity result for toric symplectic manifolds \cite[Theorem 1.1]{M}, Franz--Yamanaka \cite{FY} recently showed that the isomorphism type of a GKM graph is encoded in its graph cohomology. Hence, for a GKM manifold, the equivariant cohomology contains complete information on the combinatorics of the one-skeleton which translates to a complete understanding of the combinatorial aspects of the entire orbit type stratification, {see Remark \ref{rem:faceposet}}. Furthermore in \cite{All} and \cite{Q}, combinatorial aspects of actions of compact Lie groups are related to the spectrum of the equivariant cohomology ring, leading among other things to an algebraic criterion for uniformity of an action. {These results naturally trigger the question of how far its equivariant cohomology determines the combinatorics of a general torus action.} More specifically, we consider the connected orbit type stratification of an action of a compact torus $T$ on a space $X$, i.e.\ the collection of all connected components of fixed point sets $X^U$ where $U\subset T$ is {a} subtorus. It naturally carries the combinatorial structure of a poset as well as a function which remembers the kernel of the action on each element of the stratum. We ask whether this combinatorial data is encoded in the rational equivariant cohomology algebra. In full generality, such a statement cannot hold true, as the equivariant cohomology algebra of any torus action on a sphere with nonempty, connected fixed point set is the same as that of the trivial action, see Example \ref{ex:spheres}. In this paper we argue that the reason for such behavior is to be found in the existence of unramified elements in the orbit type stratification. We define an element in the orbit type stratification to be \emph{ramified} if it is either minimal, or, recursively, minimal with the property that it contains two distinct ramified elements -- see Definition \ref{defn:ramified}.  Our first result, Theorem \ref{thm:encodesstratification}, which holds under mild topological assumptions on the space but without any further conditions on the action, states:
 \begin{thm*} The rational equivariant cohomology of a compact $T$-space $X$ encodes the subposet $\overline{\chi}$ of ramified elements in the connected orbit type stratification, together with the function $\overline{\lambda}:\overline{\chi}\to \{\textrm{connected subgroups of }T\}$ that {sends an element $C\in\overline{\chi}$ to} the identity component of the kernel of the $T$-action on $C$.
 \end{thm*}
Our technique of choice to prove this theorem is the new notion of Thom system, which we introduce in Definition \ref{defn:thomsystem}. It may be formulated in the general context of (graded) commutative rings and encodes algebraic properties of the equivariant Thom classes of the path-components of the fixed point set in case of a smooth action. While this connection to Thom classes plays a major role in the second half of the paper, the theorem above could also be proved using the machinery developed in \cite{All} and \cite{Q}, see Remark \ref{rem:somewhatdually}.
 
In more specialized settings, the subset of ramified elements may already contain further information on the stratification. We note that the ramification condition in the statement below, which is Corollary \ref{cor:stratifequivformal}, is satisfied in particular if the fixed points of the action are isolated.
\begin{cor*}
Let $X$ be an equivariantly formal, compact $T$-space such that every isotropy codimension $1$ element of the orbit type stratification poset $\chi$ is ramified. Then $H_T^*(X)$ encodes $\chi$ up to rational equivalence. If $X$ is additionally a manifold and the $T$-action is smooth, then all of $\chi$ is encoded in $H_T^*(X)$.
\end{cor*}
 
In Section \ref{sec:cohom} we investigate how far the equivariant cohomology algebra of a $T$-space encodes all other (equivariant) cohomology algebras in the orbit type stratum. A starting point is the following result (cf.\ Proposition \ref{prop:fpdimension}).
\begin{prop*}
The sum of all Betti numbers of each indiviual path-component of $X^T$ is encoded in $H_T^*(X)$. If $X$ is equivariantly formal then the individual sums of all Betti numbers of the components of $X^U$ are encoded in $H_T^*(X)$ for every subtorus $U\subset T$.
\end{prop*}

In particular this implies the first half of the following corollary. The second part is a consequence of the theorem above and was proved first in \cite{FY}.

\begin{cor*}
For an equivariantly formal compact orientable $T$-manifold $M$ the equivariant cohomology algebra $H_T^*(M)$ encodes whether the action is of GKM type or not. In case the action is GKM, $H_T^*(M)$ also encodes the GKM graph of the action.
\end{cor*}
In general, one can not expect $H_T^*(X)$ to contain more specific information like individual Betti numbers of the strata or multiplicative structure, even for equivariantly formal actions on compact manifolds whose orbit type stratification consists only of ramified elements (see Example \ref{ex:cohomnotencoded}). However under stronger conditions, we show in Theorem \ref{thm:encodescohomology}:
\begin{thm*}
Let $M$ be an equivariantly formal, compact orientable $T$-manifold such that the map $H^*(M)\rightarrow H^*(X)$ is surjective for all components $X$ of $M^T$. Then the equivariant cohomology $H_T^*(M)$ encodes the connected orbit type stratification $\chi$ of $M$ as well as for any $C, D\in \chi$ with $C\subset D$ the respective equivariant cohomology algebras and the map $H_T^*(D)\rightarrow H_T^*(C)$ induced by the inclusion.
\end{thm*} 
The surjectivity condition is trivially satisfied in case the fixed point set of the action consists of isolated points.

In Section \ref{sec:integral} we conclude the paper with some remarks on the question which additional information on the orbit type stratification can be obtained from the cohomology by considering integral instead of rational coefficients.\\

\noindent {\bf Acknowledgements.\footnote{Declarations: other declarations requested in the submission guidelines are not applicable.} \footnote{Data sharing not applicable to this article as no datasets were generated or analysed during the current study.}} This work is part of a project funded by the Deutsche Forschungsgemeinschaft (DFG, German Research Foundation) - 452427095. The second author is grateful to the Max Planck Institute for Mathematics in Bonn for its hospitality and financial support.

\section{Preliminaries}
In this paper we consider continuous actions of a compact torus $T$. All spaces are assumed to be locally contractible, Hausdorff, and have finite-dimensional rational (singular) cohomology. In parts of Section \ref{sec:cohom} we will restrict to smooth manifolds. Cohomology is always singular with coefficients in $\mathbb{Q}$ if not specified otherwise and we usually suppress coefficients from the notation. When considering elements of a graded object, these will always be assumed to be homogeneous in absence of further specification.

 Given a $T$-space $X$, as well as a subgroup $U\subset T$, we denote by $X^U$ the set of points in $X$ fixed by $U$. Note that if $X$ is a smooth manifold and the action is smooth, then every component of $X^U$ is a smooth submanifold.
 {As $T$ is Abelian, the isotropy type $(T_p)$ of the $T$-orbit $T\cdot p$ of a point $p\in X$, i.e.\ the conjugacy class of $T_p$ in $T$, consists only of $T_p$. Moreover, the isotropy type contains the same information as the orbit type of $T\cdot p$ in the sense of \cite[Section I.4]{B}, i.e.\ the equivariant homeomorphism type of $T\cdot p$. Hence there is a bijection
\[\{\text{subgroups of T}\}\longleftrightarrow \{T\text{-orbit types}\}.\]
}
{As we consider mainly rational cohomology (with the exception of Section \ref{sec:integral}), we restrict our attention to connected subgroups $U$.}
\begin{defn}\label{defn:orbitstratification}
Let $X$ be a $T$-space. The \emph{connected orbit type stratification} of $X$ is the poset $\chi=\{\textrm{connected components of }X^U\mid U\subset T \textrm{ connected, closed subgroup}\}$ {ordered by inclusion}, together with the function $\lambda\colon \chi\rightarrow\{\textrm{connected, closed subgroups of }T\}$ such that $\lambda(C)$ is the identity component of the kernel of the restricted $T$-action on $C$ {(i.e., the identity component of the intersection of all isotropy groups of points in $C$)}. 
\end{defn} 

{We note that the function $\lambda$ is an order-reversing morphism of posets if we order both sides by inclusion.}

\begin{rem}\label{rem:faceposet}
{
Recording the geometric properties of torus actions in combinatorial objects has a long history with several different types of objects serving as a tool of bookkeeping for the combinatorial data. In case $X$ is a quasitoric manifold (see \cite[Definition 7.3.1]{BP}), the poset $\chi$ defined above is isomorphic to the face poset (note that all isotropy groups of a quasitoric manifold are connected). Furthermore $\lambda$ corresponds to the labelling function on faces emerging from the characteristic pair of $X$ (see \cite[Definition 7.3.4]{BP}). Thus in the quasitoric case the datum of the characteristic pair and the connected orbit type stratification are equivalent.}

{
In case $X$ is a GKM manifold, it is obvious from the definition that the GKM graph determines the connected orbit type stratification of the one-skeleton of the action. The GKM condition implies that this, in turn, already determines the complete connected orbit type stratification, see e.g.\ \cite[Lemma 4.1]{GKZ}. Conversely, the connected orbit type stratification of (the one-skeleton of) a GKM action determines the GKM graph up to scalars of the labelling.
}
\end{rem}

\begin{ex}\label{ex:posets}{
Consider $S^6\subset \mathbb{C}^3\oplus \mathbb{R}$ together with the $T^2$-action defined by $(s,t)\cdot (v,w,z,h)=(sv,tw,s^at^bz,h)$, for $a,b\in \mathbb{Z}$, where multiplication on the right hand side is usual complex multiplication. Below we depict the connected orbit type stratification for varying values of $(a,b)$. Inclusion relations are generated by the depicted arrows. For clarity we depict an element $C$ of $\chi$ together with its value under $\lambda$ as $(C',\lambda(C))$ where $C'$ is homeomorphic to $C$. We stress however that we do not consider the topological type of $C$ to be part of the combinatorial data when talking about the connected orbit type stratification as a (labelled) poset. The notation $\Delta S^1$ refers to the diagonal circle in $T^2$.}
\begin{center}

\begin{tikzpicture}
\node (a) at (0,0) {$(*,T)$};
\node (b) at (0,1.5) {$(*,T)$};
\node (c) at (3,0) {$(S^2,\{1\}\times S^1)$};
\node (d) at (3,1.5) {$(S^4,S^1\times\{1\})$};
\node (e) at (6,0.75) {$(S^6,\{1\})$};

\draw[thick, ->](a) to (c);
\draw[thick, ->](b) to (c);
\draw[thick, ->](a) to (d);
\draw[thick, ->](b) to (d);
\draw[thick, ->](c) to (e);
\draw[thick, ->](d) to (e);

\node (f) at (0,-1.75) {$(S^2,T)$};
\node (g) at (3,-1) {$(S^4,S^1\times \{1\})$};
\node (h) at (3,-2.5) {$(S^4,\{1\}\times S^1)$};
\node (i) at (6,-1.75) {$(S^6,\{1\})$};

\draw[thick, ->](f) to (g);
\draw[thick, ->](f) to (h);
\draw[thick, ->](g) to (i);
\draw[thick, ->](h) to (i);

\node (j) at (0,-3.75) {$(*,T)$};
\node (k) at (0,-5.25) {$(*,T)$};
\node (l) at (3,-3.5) {$(S^2,S^1\times \{1\})$};
\node (m) at (3,-4.5) {$(S^2,\Delta S^1)$};
\node (n) at (3,-5.5) {$(S^2,\{1\}\times S^1)$};
\node (o) at (6,-4.5) {$(S^6,\{1\})$};

\draw[thick, ->](j) to (l);
\draw[thick, ->](j) to (m);
\draw[thick, ->](j) to (n);
\draw[thick, ->](k) to (l);
\draw[thick, ->](k) to (m);
\draw[thick, ->](k) to (n);
\draw[thick, ->](l) to (o);
\draw[thick, ->](m) to (o);
\draw[thick, ->](n) to (o);

\node at (-3,0.75) {$(a,b)=(0,1)$};
\node at (-3,-1.75) {$(a,b)=(0,0)$};
\node at (-3,-4.5) {$(a,b)=(1,-1)$};

\draw[dashed] (-1,-6)--(-1,2);
\draw[dashed] (-6,-0.5)--(8,-0.5);
\draw[dashed] (-6,-3)--(8,-3);
\end{tikzpicture}

\end{center}
\end{ex}

To any $T$-space $X$ one associates its equivariant cohomology $H^*_T(X)$, which is the cohomology of its Borel construction $X_T:=ET\times_T X$, equipped with the structure of $H^*(BT)$-algebra induced by the natural projection $ET\times_T X\to BT$. We will abbreviate $R:=H^*(BT)$. {A $T$-equivariant map $X\rightarrow Y$ induces a map between Borel constructions and we obtain an induced map $H_T^*(Y)\rightarrow H_T^*(X)$. Additionally, for any subgroup $H\subset T$ we may set $EH=ET$ as $H$ acts freely on $ET$. Then the canonical projection $X_H\rightarrow X_T$ induces a map $H_T^*(X)\rightarrow H_H^*(X)$.}

\begin{lem}\label{lem:equivformal} For a $T$-space $X$ the following conditions are equivalent:
\begin{enumerate}[(i)]
\item $H^*_T(X)$ is a free $R$-module.
\item $H^*_T(X) = R\otimes H^*(X)$ as an $R$-module.
\item The map $H^*_T(X)\to H^*(X)$ induced by a fiber inclusion $X \to X\times_T ET$ is surjective.
\item $\dim_{\mathbb{Q}} H^*(X^T) = \dim_{\mathbb{Q}} H^*(X)$.
\end{enumerate}
\end{lem}
These equivalences are standard, see e.g.\ \cite[Corollary 4.2.3]{AP} and  \cite[Corollary IV.2]{Hs}.

\begin{defn}\label{defn:eqformal} If a $T$-space satisfies {one of} the equivalent conditions of Lemma \ref{lem:equivformal} then we say that it is equivariantly formal.
\end{defn}
Let us recall the Borel localization theorem, see e.g.\ \cite[Theorem 3.2.6]{AP}: for any multiplicatively closed subset $S\subset R$, we set
\[
X^S = \{x\in X\mid S^{-1}H_T^*(Tx)\neq 0\},
\]
where $S^{-1}$ denotes localization at $S$.

\begin{thm}
Assume that either $X$ is compact or that it is paracompact and that the set of identity components of isotropy subgroups is finite. Then the inclusion $X^S\to X$ induces an isomorphism of $S^{-1}R$-modules $S^{-1}H^*_T(X) \to S^{-1}H^*_T(X^S)$.
\end{thm}

We will be particularly interested in the following situation: for a subtorus $U\subset T$, let $\mathfrak{p}_U = \ker (H^*(BT)\to H^*(BU))$, and $S = R\setminus {\mathfrak{p}}_U$. Then $X^S = X^U$, the fixed point set of the restricted $U$-action, and we obtain an isomorphism
\begin{equation}\label{eq:borelU}
S^{-1}H^*_T(X) \longrightarrow S^{-1}H^*_T(X^U).
\end{equation}
For $U=T$ we obtain in particular that the kernel of the canonical map $H^*_T(X)\to H^*_T(X^T)$ is the {$R$-torsion} submodule of $H^*_T(X)$. Denoting by $X_1 = \{x\in X\mid \dim {Tx} \leq 1\}$ the one-skeleton of the action, we have a sequence
\begin{equation}\label{eq:shortexact}
0 \longrightarrow H^*_T(X) \longrightarrow H^*_T(X^T)\longrightarrow H^*_T(X_1,X),
\end{equation}
which, for an equivariantly formal action, is exact at $H^*_T(X)$.
\begin{lem}[Chang-Skjelbred Lemma {\cite[Lemma 2.3]{CS}}]\label{lem:CSL}
Assume that either $X$ is compact or that it is paracompact and that the set of identity components of isotropy subgroups is finite. Then for an equivariantly formal action of a torus $T$ on $X$ the sequence \eqref{eq:shortexact} is also exact at $H^*_T(X^T)$. 
\end{lem}
In the following proposition we collect some well-known properties of equivariantly formal actions.
\begin{prop}\label{prop:eqforminherited} Consider an equivariantly formal $T$-action on a space $X$. Then the following hold true:
\begin{enumerate}
\item For any subtorus $U\subset T$, the induced $U$-action on $X$ is equivariantly formal.
\item For any subtorus $U\subset T$, the induced $T$-action on every component of $X^U$ is equivariantly formal.  
\end{enumerate}
\end{prop}
\begin{proof}
We observe that for any subtorus $U\subset T$ the map $H^*_T(X)\to H^*(X)$ induced by fiber inclusion factors as $H^*_T(X)\to H^*_U(X)\to H^*(X)$. Thus, if $H^*_T(X)\to H^*(X)$ is surjective, so is $H^*_U(X)\to H^*(X)$. The first statement then follows from Lemma \ref{lem:equivformal}.

Lemma \ref{lem:equivformal} then implies that $\dim H^*(X^T) = \dim H^*(X) = \dim H^*(X^U)$. As $(X^U)^T= X^T$, this implies, again via Lemma \ref{lem:equivformal}, that the $T$-action on $X^U$ is equivariantly formal. In general, a $T$-action is equivariantly formal if and only if the action on every connected component is equivariantly formal.
\end{proof}

\section{Orbit type stratification}

The purpose of this section is to recover information on the combinatorics of the orbit type stratification of a $T$-action from the algebraic data of its equivariant cohomology algebra $H^*_T(X)$. To this end we introduce the notion of Thom system, see Definition \ref{defn:thomsystem}. It is motivated by the fact that in certain smooth settings, a preferred choice of Thom system in the equivariant cohomology will be given by the equivariant Thom classes of the path-components of the fixed point set (cf.\ Lemma \ref{lem:minimalthomsystem}). 
\begin{rem}\label{rem:somewhatdually}
Somewhat dually to this approach, Quillen \cite{Q} and Allday \cite{All}, see also \cite[Section 3.6]{AP}, related the combinatorics of the action to the spectrum of the even degree part of the equivariant cohomology ring, by considering ideals of the form ${\mathfrak{p}}(K,c) = \ker(H^*_T(X)\to H^*(BK))$ where $K\subset T$ is a subtorus, $c$ a component of $X^K$, and the map is induced by the inclusion of a point in $c$. The set of all such pairs $(K,c)$ forms a poset $\mathcal{T}(X)$, where $(K,c)\leq (L,d)$ if and only if $K\subset L$ and $d\subset c$. The results in \cite{All} can be used to reconstruct $\mathcal{T}(X)$ from $H^*_T(X)$, in some sense substituting for the roles of Theorems \ref{thm:detectthomsys} and \ref{thm:inclusiondetector} in this section. The poset $\mathcal{T}(X)$ is not the same as the connected orbit type stratification (see Definition \ref{defn:orbitstratification}) as it does not detect whether, for $(K,c)\leq (L,d)$, the inclusion $d\subset c$ is strict. Rather, this poset corresponds, in the proof of Theorem \ref{thm:encodesstratification} below, to the poset $\chi'$. As we are interested in the more geometric orbit type stratification, we are led to the concept of ramification, see Definition \ref{defn:ramified} below.
\end{rem}

\begin{defn}\label{defn:thomsystem} For a (graded) commutative ring $A$, we call a (homogeneous) collection of elements $\tau_1,\ldots,\tau_k\in A$ a Thom system if
\begin{itemize}
\item
$\tau_i\cdot \tau_j$ is nilpotent whenever $i\neq j$
\item $\tau_i$ is not nilpotent
\item for any system $\alpha_1,\ldots,\alpha_l\in A$ satisfying the two preceding properties we have $l\leq k$.
\end{itemize}
\end{defn}

\begin{ex}
\begin{enumerate}[(i)]{
\item
As a purely algebraic example consider the ring $A=\mathbb{Q}[X,Y]/(Y^2-XY)$. Then $\tau_1=Y$ and $\tau_2=X-Y$ form a Thom system when considered as elements of $A$ by abuse of notation. Indeed observe that $Y^2-XY=Y\cdot (Y-X)$ is a prime decomposition and thus $\tau_1\tau_2=0$ and $\tau_i^k\neq 0$ for $k\geq 0$, $i=1,2$. Thus the first two properties of a Thom system are satisfied. It remains to prove that the maximal size of a collection  of elements with these properties is $2$. Assume for contradiction that $\eta_1,\eta_2,\eta_3\in A$ are elements satisfying the first two properties of a Thom system.
Consider representatives $f_1,f_2,f_3\in \mathbb{Q}[X,Y]$. Since every prime factor of $Y^2-XY$ occurs with power $1$ it follows that $A$ has no nonzero nilpotent elements. Consequently $Y^2-XY$ divides $f_1f_2$ but neither $f_1$ nor $f_2$ (since $\eta_1,\eta_2$ are nonzero).
We assume without loss of generality that $Y|f_1$ and $Y-X|f_2$. But then $\eta_1\eta_3=0$ implies $Y-X|f_3$ and $\eta_2\eta_3=0$ implies $Y|f_3$. Hence $Y^2-XY|f_3$ and $\eta_3=0$
which is a contradiction.
\item With $R$ as before, consider a graph $\Gamma$ with vertices $V(\Gamma)$, (non-oriented) edges $E(\Gamma)$, and a labelling function $\alpha\colon E(\Gamma)\rightarrow H^2(BT)\backslash \{0\}$. The graph cohomology $H^*(\Gamma,\alpha)$ in the sense of \cite{GKM, GZ} is defined as the subalgebra of $R^{V(\Gamma)}$ of elements $f$ such that for any $v,w\in V(\Gamma)$ that are connected by an edge $e\in E(\Gamma)$ we have $\alpha(e)|f(v)-f(w)$. The ring structure comes from componentwise multiplication in $R^{V(\Gamma)}$.
For a vertex $v$ let $E_v$ denote the set of edges emanating from $v$. One defines as in \cite[Section 2.3]{GZ0} the Thom class of $v$ as the element $f\in H^*(\Gamma,\alpha)$ with $f(v)=\prod_{e\in E_v} \alpha(e)$ and $f(w)=0$ for $v\neq w$. Now since $R$ has no nontrivial zero-divisors one easily checks that the set of all Thom classes of vertices of $\Gamma$ defines a Thom system of $H^*(\Gamma,\alpha)$. In fact, a collection $\tau_1,\ldots,\tau_k\in H^*(\Gamma,\alpha)$ is a Thom system if and only if each $\tau_i$ is nontrivial on a unique vertex and this correspondence between the $\tau_i$ and the vertices is bijective. Setting $R=\mathbb{Q}[Z]$ the algebra $A$ from part $(i)$ can be seen as the graph cohomology $H^*(\Gamma,\alpha)$ of}
\begin{center}
\begin{tikzpicture}
\draw[very thick] (0,0) -- ++(2,0);

  \node at (0,0)[circle,fill,inner sep=2pt]{};
   \node at (2,0)[circle,fill,inner sep=2pt]{};
   
   \node at (1,0.3){Z};
   
   \node at (-1.1,0){$(\Gamma,\alpha)=$};
\end{tikzpicture}
\end{center}
{
by including $A\rightarrow H^*(\Gamma,\alpha)\subset R\times R$ via $X\mapsto (Z,Z)$, $Y\mapsto(0,Z)$. In this way we see that $\tau_1,\tau_2$ from part $(i)$ correspond to the graph theoretic Thom classes $(Z,0)$ and $(0,Z)$ of $(\Gamma,\alpha)$.}
\end{enumerate}
\end{ex}

\begin{lem}\label{lem:trivialaction'}
Let $X$ be a path-connected space with trivial $T$-action and $\alpha\in H_T^*(X)$. The following are equivalent:
\begin{enumerate}[(i)]
\item $\alpha$ is not nilpotent.
\item $\alpha$ restricts to a nontrivial element in {$H_T^*(p)$ for any point $p\in X$}.
\item multiplication with $\alpha$ is injective on $H_T^*(X)$.
\end{enumerate}
\end{lem}

\begin{proof}
We have $H_T^*(X)=R\otimes H^*(X)$ as $R$-algebras. Thus an element is nilpotent if and only if it is contained in $R\otimes H^+(X)$. This proves the equivalence of $(i)$ and $(ii)$. Clearly also $(iii)$ implies $(i)$. Finally, assume that $(ii)$ holds and write the nontrivial $R\otimes H^0(X)$-component of $\alpha$ as $f\otimes 1$. It follows that multiplying an element of $R\otimes H^{\geq k}(X)$ with $\alpha$ multiplies its $R\otimes H^k(X)$ component with $f$. Thus multiplication with $\alpha$ is injective.
\end{proof}

\begin{prop}\label{prop:TSchar'}
Let $X$ be a space with trivial $T$-action. Then $H_T^*(X)$ admits a Thom system. A collection $\tau_1,\ldots,\tau_k\in H_T^*(X)$ is a Thom system if and only if $X$ has $k$ connected components $X_1,\ldots,X_k$ which can be numbered in a way such that for any choice of points $p_i\in X_i$
\begin{itemize}
\item $\tau_i$ restricts to $0$ in $H_T^*(p_j)$ for $i\neq j$
\item the restriction of $\tau_i$ to $H_T^*(p_i)$ is not $0$.
\end{itemize}
\end{prop}

\begin{proof}
Let $\tau_1,\ldots,\tau_l\in H_T^*(X)$ be elements satisfying the first two conditions in the definition of a Thom system and let $X_1,\ldots,X_k$ be the components of $X$. Choose $p_i\in X_i$. As the $\tau_i$ are not nilpotent, Lemma \ref{lem:trivialaction'} shows that they restrict nontrivially to at least one of the $H_T^*(p_i)$. Also, since the {rings} $H_T^*(p_i)$ are integral domains it follows that no two of the $\tau_i$ restrict nontrivially to the same point. Thus it follows that $l\leq k$ and that, if $l=k$, then the $\tau_i$ correspond bijectively to the $X_i$ in the manner described in the proposition (after possibly adjusting the order). Thus it remains to argue that a Thom system has $k$ elements. This follows from the fact that the first two conditions in the definition of a Thom system are satisfied by the elements $e_1,\ldots,e_k$ defined by the condition that $e_i$ restricts to $1$ in $H_T^*(X_i)$ and to $0$ in $H_T^*(X_j)$ for $j\neq i$.
\end{proof}

For a subtorus $U\subset T$ of a torus $T$ we denote 
\[
\mathfrak{p}_U:=\ker (H^*(BT)\rightarrow H^*(BU))
,\]
as well as $S_U:=R\setminus\mathfrak{p}_U$. For a space $X$ with $T$-action, we will also consider $H^*_U(X)$ as an algebra over $R$, via the map $H^*(BT)\rightarrow H^*(BU)$.

\begin{lem}\label{lem:kernel}
Let $X$ be a compact $T$-space, and $U\subset T$ a subtorus which acts trivially on $X$. If $x\in \ker ( H_T^*(X)\rightarrow H_U^*(X))$ then it is nilpotent in $H_T^*(X)/\mathfrak{p}_UH_T^*(X)$.
\end{lem}

\begin{proof}
Let $S$ be a subtorus of $T$ which is complementary to $U$, i.e.\ $T=U\times S$. Then $R=H^*(BU)\otimes H^*(BS)$ and $H_T^*(X)\cong H^*(BU)\otimes H_S^*(X)$ as $R$-algebras with the obvious $R$-algebra structure. Furthermore $H_U^*(X)\cong H^*(BU)\otimes H^*(X)$ and the restriction $r\colon H_T^*(X)\rightarrow H_U^*(X)$ corresponds to
\[\id_{H^*(BU)}\otimes r'\colon H^*(BU) \otimes H_S^*(X)\rightarrow H^*(BU)\otimes H^*(X)\]
where $r'$ is the restriction $H_S^*(X)\rightarrow H^*(X)$. Both algebras above inherit a bigrading with respect to the tensor product. The bigrading is respected by $r$.

If $x$ lies in the kernel of $r$ then its $H^*(BU)\otimes H^0_S(X)$ component is zero. Hence $x\in H^*(BU)\otimes H^+_S(X)$. Note that for $N$ large enough, any product of $N$ elements of $H^+_S(X)$ lies in $H^+(BS)\cdot H_S^*(X)$. This follows from the fact that $H_S^*(X)$ is finitely generated as $H^*(BS)$-module \cite[Prop.\ 3.10.1]{AP} where we put $N$ larger than the highest degree in a {generating set of $H_S^*(X)$ over $H^*(BS)$}. Consequently, $x^N$ lies in $H^+(BS)\cdot H_T^*(X)$ and in particular in $\mathfrak{p}_U H_T^*(X)$.
\end{proof}

\begin{lem}\label{lem:nilpotencelem}
Let $X$ be a compact $T$-space and $U\subset T$ a subtorus. Then for any $x\in H_T^*(X)$ the image of $x$ in $H_U^*(X^U)$ is nilpotent if and only if the image of $x$ in $S_U^{-1}(H_T^*(X)/\mathfrak{p}_UH_T^*(X))$ is nilpotent.
\end{lem}

\begin{proof}
As $\mathfrak{p}_UH^*_T(X)$ is contained in the kernel of $H^*_T(X)\to H^*_U(X)$, the restriction map $H_T^*(X)\to H_U^*(X^U)$ factors as
\[
H^*_T(X)\to H^*_T(X)/\mathfrak{p}_UH^*_T(X) { \to H^*_U(X)} \to H_U^*(X^U).
\]
Applying localization at $S_U$, it follows that if the image of $x$ in $S_U^{-1}H_T^*(X)/S_U^{-1}\mathfrak{p}_UH_T^*(X)$ is nilpotent, then the same holds for the image in $S_U^{-1}H_U^*(X^U)$. An element of $H_U^*(X^U)$ is nilpotent if and only if it is nilpotent in $S_U^{-1}H_U^*(X^U)$ which proves one direction.

Assume conversely that the image of $x$ in $H_U^*(X^U)$ is nilpotent. Then some power $x^k$ maps to the kernel of $H_T^*(X^U)\rightarrow H_U^*(X^U)$. By Lemma \ref{lem:kernel}, some higher power $x^N$ satisfies $f(x^N)\in \mathfrak{p}_U H_T^*(X^U)$ where $f$ is the map $H_T^*(X)\rightarrow H_T^*(X^U)$. By Borel localization the map
\[S_U^{-1}f\colon S_U^{-1} H_T^*(X)\rightarrow S_U^{-1}H_T^*(X^U)\]
is an isomorphism of $S_U^{-1}R$-modules, see \eqref{eq:borelU}. Since $S_U^{-1}f(x^N)$ is in $S_U^{-1}\mathfrak{p}_U H_T^*(X^U)$ it follows that the image of $x^N$ in $S_U^{-1}H_T^*(X)$ lies in $S_U^{-1}\mathfrak{p}_U H_T^*(X)$. The Lemma follows from the fact that localization commutes with taking quotients, i.e.\ $S_U^{-1}H_T^*(X)/S_U^{-1}\mathfrak{p}_U H_T^*(X)\cong S_U^{-1}(H_T^*(X)/\mathfrak{p}_UH_T^*(X))$.
\end{proof}

\begin{rem}\label{rem:sphericalstuff}
With regards to the above lemma and the theorem below, we remark that in general $S_U^{-1}(H_T^*(X)/\mathfrak{p}_UH_T^*(X))$ and $S_U^{-1}H_U^*(X^U)$ are not isomorphic. Also, we only obtain criteria for elements to restrict to nilpotent elements in $H_U^*(X^U)$ without a precise description on the kernel. The reason for this is the fact that in general the kernel of the restriction $H_T^*(X)\rightarrow H_U^*(X)$ is larger than $\mathfrak{p}_U H_T^*(X)$ and may additionally contain certain Massey products involving elements of $\mathfrak{p}_U$. This phenomenon is discussed in \cite{AZ} under the name of spherical actions.
\end{rem}

\begin{thm}\label{thm:detectthomsys}
Let $X$ be a compact $T$-space and $U\subset T$ a subtorus. Then there are elements $\tau_1,\ldots,\tau_k\in H_T^*(X)$ which restrict to a Thom system of $H_U^*(X^U)$. A set of elements $\tau_1,\ldots,\tau_k\in H_T^*(X)$ has this property if and only if it restricts to a Thom system in $S^{-1}_U(H_T^*(X)/\mathfrak{p}_U H_T^*(X))$.
\end{thm}

\begin{proof}
By Lemma \ref{lem:nilpotencelem}, $\tau_1,\ldots,\tau_r\in H_T^*(X)$ satisfy the first two conditions of a Thom system when restricted to $H_U^*(X^U)$ if and only if they do so in $S_U^{-1}(H_T^*(X)/\mathfrak{p}_U H_T^*(X))$. Let $\eta_1,\ldots,\eta_l\in S_U^{-1}(H_T^*(X)/\mathfrak{p}_U H_T^*(X))$ be a set satisfying the nilpotence conditions of a Thom system. Then multiplying the $\eta_i$ with elements of $S_U$ preserves these conditions. Consequently, we may assume that the $\eta_i$ are restrictions from $H_T^*(X)$. Since $H_U^*(X^U)$ admits a Thom system by Proposition \ref{prop:TSchar'}, it follows that $l\leq k$, where $k$ is the number of elements in a Thom system of $H_U^*(X^U)$, i.e., the number of path components of $X^U$. In particular $S_U^{-1}(H_T^*(X)/\mathfrak{p}_U H_T^*(X))$ admits a Thom system which lies in the image of $H_T^*(X)\rightarrow S_U^{-1}(H_T^*(X)/\mathfrak{p}_U H_T^*(X))$.

The theorem follows if we show that the image of $H_T^*(X)\rightarrow H_U^*(X^U)$ contains a Thom system of $H_U^*(X^U)$. Let $X_1,\ldots,X_k$ be the components of $X^U$ and let $e_i\in H_T^*(X^U)$ be the element which restricts to $1\in H_T^*(X_i)$ and to $0\in H_T^*(X_j)$ for $i\neq j$. By Borel Localization there are polynomials $f_i\in S$ such that $f_i\cdot e_i$ is the restriction of some $\tau_i\in H_T^*(X)$. It remains to check that the $\tau_i$ restrict to a Thom system of $H_U^*(X^U)$. Choose points $p_i\in X_i$. The composition
\[H_T^*(X)\rightarrow H_T^*(X^U)\rightarrow H_U^*(X^U)\rightarrow H_U^*(p_i)\]
is $R$-linear. Thus $\tau_i$ maps to $f_i\cdot 1\in H_U^*(p_i)$, which is nonzero because $f\in S_U$, and to $0\in H_U^*(p_j)$ for $i\neq j$. Thus the $\tau_i$ restrict to a Thom system in $H_U^*(X^U)$ by Lemma \ref{lem:trivialaction'}.
\end{proof}

\begin{defn}
For a subtorus $U\subset T$, we call a set of elements $\tau_1,\ldots,\tau_k\in H_T^*(X)$ with the property as in Theorem \ref{thm:detectthomsys} a \emph{$U$-local Thom system} (of $H_T^*(X)$). Given such a system, we denote by $F_U(\tau_i)$ the unique component of $X^U$ such that $\tau_i$ restricts to a nonzero element in $H_U^*(p_i)$ for any point $p_i\in F_U(\tau_i)$.
\end{defn}

\begin{thm} \label{thm:inclusiondetector}
Let $X$ be a compact $T$-space and $H\subset U\subset T$ subtori. Let $\tau_1,\ldots,\tau_k\in H_T^*(X)$ be a $U$-local Thom system and $\eta_1,\ldots,\eta_l$ be an $H$-local Thom system. Then {for any $i,j$},  $F_U(\tau_i)\subset F_H(\eta_j)$ if and only if there is $f\in S_H$ such that the image of $f\tau_i-\eta_j\tau_i$ in $S^{-1}_U(H_T^*(X)/\mathfrak{p}_U H_T^*(X))$ is nilpotent.
\end{thm}

\begin{proof}
For $i=1,\ldots,k$, let $p_i\in F_U(\tau_i)$ and let $r_i\colon H_T^*(X)\rightarrow H_U^*(p_i)$ be the natural restriction map. We claim that for any $x\in H_T^*(X)$ and $f\in R$ we have $r_i(x)=f+\mathfrak{p}_U\in R/\mathfrak{p}_U\cong H_U^*(p_i)$ if and only if the image of $f\tau_i-x\tau_i$ in $S^{-1}_U(H_T^*(X)/\mathfrak{p}_U H_T^*(X))$ is nilpotent. To prove the claim, recall that by Lemma \ref{lem:nilpotencelem} the latter condition is equivalent to $f\tau_i-x\tau_i$ being nilpotent in $H_U^*(X^U)$. By Lemma \ref{lem:trivialaction'} this is again equivalent to $r_j(f\tau_i-x\tau_i)$ being $0$ for $j=1,\ldots,k$. Since $r_j(\tau_i)=0$ for $j\neq i$ this depends only on $r_i(f\tau_i-x\tau_i)$ being zero. But this is the case if and only if $r_i(f\cdot 1)=r_i(x)$, which proves the claim.

The inclusion $F_U(\tau_i)\subset F_H(\eta_j)$ holds if and only if $p_i\in F_H(\eta_j)$. This is the case if and only if the image of $\eta_j$ in $H_H^*(p_i)$ is not $0$. Since this restriction map factors through $H_U^*(p_i)$ the condition is equivalent to $r_i(\eta_j)\notin \ker(H^*(BU)\rightarrow H^*(BH))=\overline{\mathfrak{p}_H}$, where the latter denotes the image of $\mathfrak{p}_H$ in $H^*(BU)$. This is equivalent to $r_i(\eta_j)=r_i(f\cdot 1)$ for some $f\in S_H$, proving the theorem.
\end{proof}

{We will try to use the above results to reconstruct the connected orbit type stratification from the equivariant cohomology. However there are clearly limitations to this approach as illustrated by the following}

\begin{ex}\label{ex:spheres}
Consider a $T$-action on a sphere $S^n$ with nonempty, connected fixed point set $F$ {(e.g., Example \ref{ex:posets} with $(a,b)=(0,0)$}). Then the action is automatically equivariantly formal -- for even $n$ any action on $S^n$ is {as then cohomology is concentrated in even degrees and consequently the Serre spectral sequence of the Borel fibration collapses on the second page. For odd $n$ this follows from Lemma \ref{lem:equivformal} (iv) because then the total Betti number of $F$ is necessarily $2$: the total Betti number  of the fixed point set $X^T$ of a torus action on a space $X$ is bounded from above by the total Betti number of $X$ \cite[Corollary 3.1.14]{All}, hence $0<\dim H^{*}(F)\leq 2$. Moreover \cite{Kob}, there is an equality of Euler characteristics $\chi(X^T)=\chi(X)$, hence $\dim H^*(F)=2$. }

We choose an $R$-module basis of $H_T^*(S^n)$ of the form $\{1,a\}$, with $a\in H^n_T(S^n)$. By replacing $a$ by an element of the form $a+f$, with $f\in H^n(BT)$, we may assume that $a$ restricts to an element in $R\otimes H^+(F)$. As the restriction map $H_T^*(S^n) \to H_T^*(F) = R\otimes H^*(F)$ is injective, this implies that $a^2=0$ (as $H^+(F)$ is concentrated in only one degree).

We have shown that the equivariant cohomology $H_T^*(S^n)$ is, as an $R$-algebra, isomorphic to that of the trivial $T$-action on $S^n$. In particular, equivariant cohomology can not distinguish these actions. However, among those indistinguishable actions, many different orbit type stratifications are possible.
\end{ex}

Let $\chi$ be the connected orbit type stratification, see Definition \ref{defn:orbitstratification}. {To make precise the capabilities and limits of Theorems \ref{thm:detectthomsys} and \ref{thm:inclusiondetector}} let us introduce the following recursive

\begin{defn} \label{defn:ramified}
{
We begin by declaring an element $C\in \chi$ as \emph{ramified} if it is minimal in the poset $\chi$. Then in the next step, we declare an element $C\in \chi$ as ramified if there exist two distinct elements $D_1\neq D_2\in\chi$ which were previously declared as ramified and which satisfy
\begin{itemize}
\item $D_1,D_2\subset C$
\item $C$ is minimal with respect to the above property in the sense that for any $C'\in\chi$ with $D_1,D_2\subset C'$ we have $C\subset C'$.
\end{itemize}
This step is repeated until after finitely many iterations no new ramified elements arise.}
\end{defn}

{Note that the process in the above definition indeed terminates after finitely many steps because the length of chains in $\chi$ is bounded.}

\begin{defn} We define $\overline{\chi}$ to be the subposet of $\chi$ given by all ramified elements in $\chi$.
\end{defn}
{
\begin{ex}\label{ex:sphere-ramification}
For the orbit type stratifications depicted in Example \ref{ex:posets} we have $\chi=\overline{\chi}$ in cases $(a,b)=(0,1)$ and $(a,b)=(1,0)$. To see this note that the two distinct fixed points $(*,T)$ in the left hand column are ramified as they are minimal. Also each of the elements in the central column is minimal with respect to the property that it contains the two fixed points and is therefore also ramified. The (maximal) element in the right hand column is ramified because it is minimal with respect to the property that it contains any two (distinct and ramified) elements in the central column.
In the case $(a,b)=(0,0)$ only the unique minimal element in the left hand column is ramified. This is due to the fact that after declaring the minimal elements as ramified, there is only a single ramified element. Hence no new elements are declared as ramified when looking for elements which are minimal with respect to the property of containing two distinct ramified elements. More generally this argument shows that for any torus action with nonempty, connected fixed point set the poset $\overline{\chi}$ of ramified elements consists only of a single element. Thus while $\chi$ may vary among the indistinguishable actions described in Example \ref{ex:spheres}, the ramified poset $\overline{\chi}$ does not.
\end{ex}
}

\begin{thm}\label{thm:encodesstratification} The equivariant cohomology of a compact $T$-space $X$ encodes the subposet $\overline{\chi}$ of ramified elements in the poset $\chi$ of orbit type strata, together with the restriction {$\left.\lambda\right|_{\overline{\chi}}:\overline{\chi}\to \{\textrm{connected subgroups of }T\}$} of $\lambda$.
\end{thm}
\begin{proof}
We construct a poset $\overline{\chi}'$ and a map $\lambda'\colon \overline{\chi}'\to \{\textrm{connected subgroups of }T\}$ together with an isomorphism $\varphi\colon \overline{\chi}'\rightarrow \overline{\chi}$ of posets satisfying $\lambda'=\lambda\circ\varphi$.
We fix a $U$-local Thom system $\tau^U_1,\ldots,\tau^U_{k_U}\in H_T^*(X)$, for every subtorus $U\subset T$, and define $\chi'$ to be the set of tuples $(\tau_i^U,U)$.
We write $(\tau_j^U,U)\leq (\tau_i^H,H)$ whenever $H\subset U$ and $f\tau_j^U-\tau_i^H\tau_j^U$ is nilpotent in $S^{-1}_U(H_T^*(X)/\mathfrak{p}_U H_T^*(X))$ for some $f\in S_H$. Then by Theorem \ref{thm:inclusiondetector} this turns $\chi'$ into a partially ordered set which corresponds bijectively to the poset of pairs $(C,U)$, where $C$ is a component of $X^U$.
The map $(\tau_i^U,U)\mapsto F_U(\tau_i^U)$ corresponds to the forgetful map $(C,U)\mapsto C$ and gives a surjection $\varphi\colon\chi'\rightarrow \chi$ compatible with the poset structure.

In analogy with $\overline{\chi}$, we call an element $C\in \chi'$ ramified if it is either minimal in $\chi'$ or there exist two ramified elements $D_1\neq D_2$ in $\chi'$ with the property that $C$ is minimal among the elements containing $D_1$ and $D_2$. Furthermore we set $\lambda'(\tau_i^U,U)=U$. We claim that $\varphi$ restricts to a bijection $\overline{\chi}'\rightarrow\overline{\chi}$ between ramified subsets and that $\lambda'=\lambda\circ \varphi$ on $\overline{\chi}'$.
Note first that for any $C\in \chi'$ we have $\codim \lambda'(C)\geq \codim \lambda(\varphi(C))$ and that any element $C\in\overline{\chi}'$ has to satisfy $\lambda'(C)=\lambda (\varphi(C))$ due to the minimality condition. Also if $\varphi(C)=\varphi(D)$ for $C,D\in \overline{\chi}'$ then in particular $\lambda'(C)=\lambda'(D)$. But then the properties of Thom systems yield $C=D$ proving injectivity of $\varphi|_{\overline{\chi}'}$. It remains to prove that $\varphi(\overline{\chi}')=\overline{\chi}$.

To see this we use induction over the isotropy codimension, where the statement is obvious for the fixed points in isotropy codimension $0$. Let $\overline{\chi}'^k$ (resp.\ $\overline{\chi}^k$) be the subset consisting of those elements $x$ for which the codimension of $\lambda'(x)$ (resp.\ $\lambda(x)$) is $k$ or less. Suppose we have shown that $\varphi(\overline{\chi}'^k)=\overline{\chi}^k$. Let $C\in \overline{\chi}^{k+1}$ with $\lambda(C)$ of codimension $k+1$. Then there {exists} $C'\in \chi'$ with $\varphi( C')=C$ and $\codim \lambda'(C')=k+1$. We show that $C'$ is ramified and hence $\varphi(\overline{\chi}')\supset \overline{\chi}$. If $C$ is minimal and $D'\leq C'$ then $\varphi(D')=C$. Thus $\codim \lambda'(D')\geq\codim\lambda(C)=\codim \lambda'(C')$ which implies $D'=C'$. If $C$ is not minimal then there are $D_1,D_2\in \overline{\chi}^k$ such that $C$ is minimal among the elements containing $D_1$ and $D_2$. Let $D_1',D_2'$ denote preimages in $\overline{\chi}'^k$. Then we have $D_1',D_2'\leq C'$. For any $D_1',D_2'\leq B'\leq C'$ we have that $\varphi(B')=C$ so $\codim\,  \lambda' (B')\geq \codim\, \lambda (C)=\codim\, \lambda'(C')$ and thus $B'=C'$. This concludes the proof of $\varphi(\overline{\chi}')\supset \overline{\chi}$. In a similar fashion, the induction proves $\varphi(\overline{\chi}')\subset \overline{\chi}$.
\end{proof}

{Every element $D\in \chi$ contains a unique maximal ramified element $C$ (i.e.\ any ramified $C'\subset D$ satisfies $C'\subset C$): first, $D$ contains a minimal element. In particular it contains a maximal ramified element $C\subset D$. If there are two distinct ramified $C,C'\subset D$ which are maximal with respect to these properties then $D$ would by definition be ramified itself, resulting in a contradiction. Thus $C$ is unique. Sending $D\mapsto C$ defines a retraction (i.e., a left-inverse)
\[
r:\chi\longrightarrow \overline{\chi}
\]
of the inclusion $\overline{\chi}\to \chi$. The retraction respects the poset structures but is not necessarily compatible with $\lambda$.
}

{
As we can reconstruct $(\overline{\chi},\lambda|_{\overline{\chi}})$ from $H_T^*(X)$ up to isomorphism of (labelled) posets, it is natural to ask how close $\overline{\chi}$ is to $\chi$. One reasonable way to quantify this is to investigate the inclusion $r(D)\subset D$. For instance if it is the identity, then $\overline{\chi}=\chi$ and we can reconstruct $\chi$ entirely from $H_T^*(X)$. A weaker notion of equivalence is given by the following
}

\begin{defn}
A map between two spaces is called a rational equivalence if it induces an isomorphism {between their rational cohomology algebras}.
\end{defn}

\begin{rem}
The existence of a rational equivalence $X\rightarrow Y$ between two spaces is much stronger than the condition $H^*(X)\cong H^*(Y)$. Under appropriate conditions on the fundamental groups it implies that $X$ and $Y$ have the same rational homotopy type and in particular isomorphic homotopy groups up to torsion.
\end{rem}

{By the isotropy codimension of an element $C\in \chi$ we mean the codimension of the subgroup $\lambda(C)\subset T$.}

\begin{cor}\label{cor:stratifequivformal}
Let $X$ be an equivariantly formal, compact $T$-space such that every isotropy codimension $1$ element of $\chi$ is ramified. Then $H_T^*(X)$ encodes $\chi$ up to rational equivalence in the sense that {it encodes $\overline{\chi}$ and that} for any $D\in \chi$ the inclusion {$r(D)\subset D$} of the unique maximal ramified element in $D$ is a rational equivalence. If $X$ is additionally a manifold and the $T$-action is smooth, then {$\chi=\overline{\chi}$ and} all of $\chi$ is encoded in $H_T^*(X)$.
\end{cor}

\begin{proof}
{By Theorem \ref{thm:encodesstratification} $H_T^*(X)$ encodes $(\overline{\chi},\lambda|_{\overline{\chi}})$. Thus it remains to show that for any $D\in\chi$ the inclusion $r(D)\subset D$ is a rational equivalence (resp.\ the identity in case $X$ is a manifold).
}

Since every isotropy codimension $1$ element is ramified, for every $D\in \chi$, the elements {$r(D)$} and $D$ have the same one-skeleton. By Proposition \ref{prop:eqforminherited}, {$r(D)$} and $D$ are both equivariantly formal so the Chang-Skjelbred Lemma \ref{lem:CSL} implies that the inclusion is a rational equivalence.

If $X$ is a manifold, we observe that for {any} $N\in \chi$ we have $N^T\neq\emptyset$ and that for {any} $p\in N^T$ the isotropy $T$-representation of $X$ at $p$ decomposes into $2$-dimensional subrepresentations of orbit dimension $1$ and the tangent space of $X^T$. We deduce that the dimension of $N$ is determined by the one-skeleton of $N$. In particular in the above setting ${r(D)}$ and $D$ are submanifolds of the same dimension so ${r(D)}=D$.
\end{proof}

{
\begin{rem}
We will argue in Lemma \ref{lem:ramified} that the requirements of Corollary \ref{cor:stratifequivformal} are fulfilled if $X$ is an equivariantly formal, compact, orientable $T$-manifold such that for the path components $X_1,\ldots,X_k$ of $X^T$ the map $H^*(X)\rightarrow H^*(X_i)$ is surjective for $1\leq i\leq k$. In particular this includes the case when $X$ is a quasitoric manifold as then the $X_i$ are just points. In combination with Remark \ref{rem:faceposet} this proves that the characteristic pair of the quasitoric manifold $X$ is encoded in $H_T^*(X)$. It then follows from \cite[Proposition 1.8]{DJ} that two quasitoric manifolds with isomorphic equivariant cohomology algebras are equivariantly homeomorphic. This was first proved in \cite[Theorem 4.1]{M}. In order to pass from algebra to combinatorics the proof of \cite[Theorem 4.1]{M} focuses on the Thom classes of codimension $2$ submanifolds which is in some sense dual to our approach of characterizing Thom classes of fixed point sets.
\end{rem}
}

\section{Cohomology}\label{sec:cohom}

In this section we discuss under which conditions equivariant cohomology contains cohomological information about the elements in the connected orbit type stratification. We continue to use the notation from the last section, i.e., for a subtorus $U\subset T$, $\mathfrak{p}_U=\ker(H^*(BT)\rightarrow H^*(BU))$ and $S_U = R\backslash \mathfrak{p}_U$.

\begin{defn}
We call a Thom system $\tau_1,\ldots,\tau_k$ \emph{strict}, if $\tau_i\tau_j=0$ for $i\neq j$. For {a} subtorus $U\subset T$, a collection $\tau_1,\ldots,\tau_k\in H_T^*(X)$ is called a \emph{strict} $U$-local Thom system of $H_T^*(X)$ if the $\tau_i$ restrict to a strict Thom system of $H_U^*(X^U)$.
\end{defn}

{Any Thom system in a commutative ring can be turned into a strict Thom system by raising all of its elements to sufficiently large powers.} However, not all statements on Thom systems transfer directly to their strict counterparts without imposing any additional conditions. The following lemma records some analogous properties.

\begin{lem}\label{lem:strictthomsys}
Let $X$ be a compact $T$-space, $U\subset T$ a subtorus, and $\tau_1,\ldots,\tau_k\in H_T^*(X)$.
\begin{enumerate}[(i)]
\item The collection $\tau_1,\ldots,\tau_k$ forms a strict $U$-local Thom system if and only if every $\tau_i$ restricts to $0\in H_U^*(X_j)$ for all components $X_j\subset X^U$ except for a single $X_i$ where it is not nilpotent.
\item The $\tau_i$ form a strict $T$-local Thom system if and only if they induce a strict Thom system of $S_T^{-1} H_T^*(X)$. If the action is equivariantly formal then this is the case if and only if the $\tau_i$ are a strict Thom system of $H_T^*(X)$.
\item Assume that the action is equivariantly formal. Then the $\tau_i$ are a strict $U$-local Thom system if and only if they induce a strict Thom system of $S^{-1}_U(H_T^*(X)/\mathfrak{p}_U H_T^*(X))$.
\end{enumerate}
\end{lem}

\begin{proof}

For the first statement, recall from Lemma \ref{lem:trivialaction'} that multiplication with $\tau_j$ is injective on $H_U^*(X_j)$. Thus $\tau_i\tau_j=0$ implies that $\tau_i$ restricts to $0$ on $H^*_U(X_j)$.

The first half of $(ii)$ follows from the localization theorem and the fact that $H_T^*(X^T)\rightarrow S_T^{-1}(X^T)$ is injective. The second half is due to the fact that $H_T^*(X)\rightarrow S_T^{-1}H_T^*(X)$ is injective in the equivariantly formal case.

Regarding statement $(iii)$, if $X$ is equivariantly formal, then we claim that $H_U^*(X)\cong H_T^*(X)/\mathfrak{p}_U H_T^*(X)$ as $H^*(BU)\cong R/\mathfrak{p}_U$-algebras. To see this, recall that any set of elements $x_i\in H_T^*(X)$ which restricts to a $\mathbb{Q}$-basis of $H^*(X)$ gives an $R$-basis of $H_T^*(X)$. Since the restricted $U$-action is again equivariantly formal by Proposition \ref{prop:eqforminherited}, the restriction of the $x_i$ to $H_U^*(X)$ is an $H^*(BU)$-basis. Consequently the restriction $H_T^*(X)\rightarrow H_U^*(X)$ is surjective with kernel $\mathfrak{p}_U H_T^*(X) $, which proves the claim.
{Now if we localize the isomorphism $H_U^*(X)\cong H_T^*(X)/\mathfrak{p}_U H_T^*(X)$, we obtain isomorphisms
\[S_U^{-1}(H_T^*(X)/\mathfrak{p}_U H_T^*(X)) \cong S_U^{-1}H^*_U(X)\cong S_U^{-1}H_U^*(X^U)\]
where the second isomorphism follows from Borel localization in the following fashion: applied to the restricted $U$-action Borel localization gives an isomorphism $S^{-1}H_U^*(X)\rightarrow S^{-1}H_U^*(X^U)$ if we set $S= H^*(BU)\backslash\{0\}$. But the $R$-module structure on $H_U^*(X)$ is understood via the map $R=H^*(BT)\rightarrow H^*(BU)$ and this, by definition of $S_U$, maps $S_U$ onto $S$. Thus we obtain the second isomorphism of localized $R$-modules.
} 

The $\tau_i$ are a $U$-local Thom system if and only if they restrict to a Thom system of $S_U^{-1}H_U^*(X^U)$ thus $(iii)$ follows.
\end{proof}

It follows from the above lemma that we can algebraically detect strict $U$-local Thom systems if the action is equivariantly formal or $U=T$. In this case the equivariant cohomology algebra encodes the total Betti numbers in the following way:

\begin{prop}\label{prop:fpdimension}
Let $X$ be a compact $T$-space, $U\subset T$ a subtorus, and $\tau_1,\ldots,\tau_k\in H_T^*(X)$ a strict $U$-local Thom system corresponding to the components $F_U(\tau_1),\ldots,F_U(\tau_k)$ of $X^U$. Assume further that either $X$ is additionally equivariantly formal or that $U=T$. Then \[\dim H^*(F_U(\tau_i))=\rk_{S_U^{-1}R/\mathfrak{p}_U} I_i\] where $I_i=\{ x\in S_U^{-1}(H_T^*(X)/\mathfrak{p}_U H_T^*(X))\mid x\tau_j=0, \text{ for all } j\neq i\}$.
\end{prop}

\begin{proof}
As argued in the proof of Lemma \ref{lem:strictthomsys}, we have
\[S_U^{-1}(H_T^*(X)/\mathfrak{p}_U H_T^*(X)) \cong S_U^{-1}H_U^*(X^U)\]
if $U=T$ or if the action is equivariantly formal. 
The ideal $I_i$ corresponds to the kernel of the restriction $S_U^{-1}H_U^*(X^U)\rightarrow S_U^{-1}H_U^*(X^U-F_U(\tau_i))$. Since $S_U^{-1}H_U^*(X^U)=S_U^{-1}H_U^*(X^U-F_U(\tau_i))\oplus S_U^{-1}H_U^*(F_U(\tau_i))$. It follows that $I_i$ is isomorphic to $S_U^{-1}H_U^*(F_U(\tau_i))$ which, as an $S_U^{-1}R/\mathfrak{p}_U$-module, is isomorphic to $S_U^{-1}R/\mathfrak{p}_U\otimes H^*(F_U(\tau_i))$.
\end{proof}

The second statement of the following corollary was first proven in \cite[Theorem 5.1]{FY}. Recall that an equivariantly formal compact orientable $T$-manifold is of GKM type \cite{GKM} if it has only finitely many fixed points, and the one-skeleton $M_1$ of the action is a finite union of $T$-invariant $2$-spheres.
\begin{cor}
For an equivariantly formal compact orientable $T$-manifold $M$ the equivariant cohomology algebra $H_T^*(M)$ encodes if the action is of GKM type or not. In case the action is GKM, $H_T^*(M)$ also encodes the GKM graph of the action.
\end{cor}
\begin{proof}
By Proposition \ref{prop:fpdimension}, in the situation at hand the equivariant cohomology algebra encodes if the fixed point set consists of isolated points. If all fixed point components are isolated points, then every isotropy codimension $1$ element of the orbit type stratification $\chi$ is ramified, and Theorem \ref{thm:encodesstratification} implies that the equivariant structure of the one-skeleton $M_1$ of the action is encoded in $H_T^*(M)$.
\end{proof}

\begin{ex}\label{ex:cohomnotencoded}
In general, we can not expect $H_T^*(X)$ to encode more specific information about the {(equivariant)} cohomology {of elements of $\chi$}, even if the action is equivariantly formal: consider the $S^1$-action on $S^4\subset \mathbb{C}\oplus \mathbb{C}\oplus\mathbb{R}$ given by $s\cdot (v,w,h)=(sv,w,h)$. As argued in Example \ref{ex:spheres}, the equivariant cohomology of this action agrees with the one of the trivial action on $S^4$. However the fixed point sets of the two actions are $S^2$ and $S^4$, which have different cohomologies.

With regards to our previous results it also seems reasonable to ask for an example where additionally the entirety of $\chi$ is ramified. Note that this is fulfilled for any $S^1$-action with at least $2$ fixed point components. We obtain such examples from the previous ones by considering the diagonal action on $S^4\times S^2$ with standard rotation on the right hand side. For these actions we obtain the cohomologically distinct fixed point sets $S^2\coprod S^2$ and $S^4\coprod S^4$.
\end{ex}

The previous examples show that additional requirements are needed to reconstruct the cohomologies in the orbit type stratification from the global equivariant cohomology. Our main result with regards to cohomology is

\begin{thm}\label{thm:encodescohomology}
Let $M$ be an equivariantly formal, compact orientable $T$-manifold such that the map $H^*(M)\rightarrow H^*(X)$ is surjective for all components $X$ of $M^T$. Then the equivariant cohomology $H_T^*(M)$ encodes the connected orbit type stratification $\chi$ of $M$ as well as for any $C, D\in \chi$ with $C\subset D$ the respective equivariant cohomology algebras and the map $H_T^*(D)\rightarrow H_T^*(C)$ induced by the inclusion.
\end{thm}

\begin{ex}We remark that the surjectivity condition is automatically satisfied in case the fixed point set is finite. Let us describe two classes of examples with nondiscrete fixed point set for which it is fulfilled. Recall first that whenever a $T$-action on a compact manifold $M$ admits a $T$-invariant Morse-Bott function $f:M\to {\mathbb{R}}$ with critical set $M^T$, then the action is equivariantly formal, and for the global minimum $c$ of $f$, the restriction map $H^*_T(M)\to H^*_T(f^{-1}(c))$ is surjective -- this follows from the arguments in \cite[Section 1]{AB}, see also \cite[Theorem 7.1]{GT}. 

Now, consider a Hamiltonian $T$-action on a compact symplectic manifold for which every component of the fixed point set is mapped, via the momentum map, to the boundary of the momentum polytope. In this setting, for any such component $X$, one can choose a component of the momentum map which attains its global minimum exactly at $X$. Thus, such $T$-actions satisfy the assumptions of Theorem \ref{thm:encodescohomology}.

Similarly, given a toric symplectic manifold $M$, with acting torus $T$, the restriction of the $T$-action to any subtorus $U$ fulfills the same assumptions. Indeed, every component of $M^U$ corresponds to a face in the momentum polytope of the $T$-action, so that it occurs as the minimum of an  appropriate component of the $T$-momentum map.
\end{ex}

Before we come to the proof of Theorem \ref{thm:encodescohomology}, we need several lemmas. See e.g.\ \cite{GT} for the notion of Cohen-Macaulay module in the context of equivariant cohomology. {Briefly, the equivariant cohomology $H_T(X)$ of a $T$-space $X$ is a finitely generated graded module over the graded polynomial algebra $H^*(BT)$. The depth of such a graded module is defined as the maximal lengh of a regular sequence of elements in the unique maximal graded ideal $H^{>0}(BT)$. As usual, one defines the (Krull) dimension as the Krull dimension of the quotient $H^*(BT)$ and the annihilator of the graded module. The depth is always bounded from above by the dimension; if equality holds we call the module a Cohen-Macaulay module.}

 Also recall that a $T$-equivariant vector bundle $V\rightarrow X$, i.e.\ a vector bundle with a $T$-action such that the transformations between fibers are linear, induces a vector bundle $V_T\rightarrow X_T$ over the Borel construction. The equivariant characteristic classes of $V$ are defined as the regular characteristic classes of $V_T$ in $H_T^*(X)$.

\begin{lem}\label{lem:spherebundle}
Let $M$ be an equivariantly formal, compact $T$-manifold and $N\subset M$ a connected component of $M^S$ for some subtorus $S\subset T$. Let $SN$ be the sphere bundle of the normal bundle of $N\subset M$ and $e\in H_T^*(N)$ its equivariant Euler class. Then the bundle induces an isomorphism $H_T^*(SN)\cong H_T^*(N)/(e)$. Furthermore, $H_T^*(SN)$ is Cohen-Macaulay.
\end{lem}

\begin{proof}
There is a fiber bundle $S^n\rightarrow (SN)_T\rightarrow N_T$ of Borel constructions. In the associated Serre spectral sequence, the generator of $H^n(S^n)$ transgresses onto the equivariant Euler class $e\in H_T^*(N)$ of the normal bundle. Now multiplication with $e$ is injective in $H_T^*(N)$: to see this it suffices to check that multiplication is injective on $H_T^*(N^T)$ and this is the case because $e$ restricts to a nonzero element in $H_T^*(*)$ for any point in $N^T$ (there it restricts to the monomial over all weights of the normal representation at $*$, see e.g.\ \cite[Lemma 6.10]{Kaw}).

As a consequence, in the spectral sequence we are ultimately left with a single row and  $H_T^*(SN)\cong H_T^*(N)/(e)$ as $R$-algebras. By \cite[Lemma 5.2]{GT} we have
\[\mathrm{depth}(H_T^*(SN))\geq \dim T-1.\] But also $\dim_R H_T^*(SN)\leq \dim T-1$ as $(0)\subsetneq \mathrm{Ann}_R(H_T^*(SN))$ and $(0)$ is prime.
\end{proof}

\begin{lem}\label{lem:divisiblebye}
Let $M$ be an equivariantly formal, compact $T$-manifold and $N\subset M$ a connected component of $M^S$ for some subtorus $S\subset T$. Let $e$ be the equivariant Euler class of the normal bundle of $N\subset M$. Then an element $x\in H_T^*(N)$ is divisible by $e$ if and only if for any component $X\subset N^T$ the restriction $x|_{H_T^*(X)}$ is divisible by $e|_{H_T^*(X)}$.
\end{lem}

\begin{proof}
{By Lemma \ref{lem:spherebundle}}, $x$ is divisible by $e$ if and only if it restricts to $0$ in $H_T^*(SN)$. As the latter is Cohen-Macaulay of dimension $\dim T-1$, this is the case if and only if the restriction to $H_T^*((SN)_1)$ is $0$, {where $(SN)_1$ denotes the subspace of $SN$ of all points whose orbits are of dimension $\leq 1$} (see \cite[Theorem 6.1]{GT}). Now $(SN)_1$ is contained in the restriction $\left.(SN)\right|_{N^T}$ of $SN$ to $N^T$. If $x|_{H_T^*(X)}$ is divisible by $e|_{H_T^*(X)}$ for every component $X\subset M^T$ then it follows that $x|_{H_T^*(\left.(SN)\right|_{N^T})}=0$ and thus $x|_{H_T^*((SN)_1)}=0$.
\end{proof}

\begin{lem}\label{lem:complementef}
Let $M$ be an equivariantly formal, compact orientable $T$-manifold such that the map $H^*(M)\rightarrow H^*(N)$ is surjective for some component $N$ of $M^T$. Then the $T$-action on $M\backslash N$ is equivariantly formal.
\end{lem}

\begin{proof}
By Lemma \ref{lem:equivformal}, any $T$-space $X$ with finite dimensional cohomology is equivariantly formal if and only if the sums over all Betti numbers satisfy $\dim_\mathbb{Q} H^*(X)=\dim_\mathbb{Q} H^*(X^T)$. Since $(M\backslash N)^T= M^T\backslash N$ it suffices to prove that $\dim H^*(M\backslash N)=\dim H^*(M)-\dim H^*(N)$. By assumption, the inclusion of $N$ is injective on homology and thus the long exact homology sequence of the pair $(M,N)$ splits into short exact sequences
\[0\rightarrow H_*(N)\rightarrow H_*(M)\rightarrow H_*(M,N)\rightarrow 0.\]
But by Lefschetz duality, applied to $M$ with a tubular neighborhood of $N$ removed, we have $H_*(M,N)=H^{n-*}(M\backslash N)$, where $n$ is the dimension of $M$. 
\end{proof}

\begin{lem}\label{lem:minimalthomsystem}
Let $M$ be an equivariantly formal, compact orientable $T$-manifold such that the map $H^*(M)\rightarrow H^*(X_i)$ is surjective for all components $X_1,\ldots,X_k$ of $M^T$. Then there is a strict Thom system $\tau_1,\ldots,\tau_k$ of $H_T^*(M)$ which is minimal in the sense that for any other strict Thom system $\tau_1'\ldots,\tau_k'$, after possibly changing the order, the cohomological degrees satisfy $|\tau_i|\leq |\tau_i'|$. It is, up to scalars from $\mathbb{Q}^\times$, uniquely given by the equivariant Thom classes of the components of $M^T$ in $M$.
\end{lem}

\begin{proof}
There is a Mayer-Vietoris sequence
\[0\rightarrow H_T^*(M)\rightarrow H_T^*(M\backslash X_i)\oplus H_T^*(X_i)\rightarrow H_T^*(SX_i)\rightarrow 0\]
where $SX_i$ is the sphere bundle of the normal bundle of $X_i$ in $M$. Let $e_i\in H_T^*(X_i)$ denote the equivariant Euler class of $SX_i$ and let $\tau_i\in H_T^*(M)$ be the equivariant Thom class of $X_i\subset M$, i.e.\ the unique element that restricts to $(0,e_i)$ in the above sequence. Then by Lemma \ref{lem:strictthomsys}, $\tau_1,\ldots,\tau_k$ form a strict Thom system.

Assume {$\tau_1',\ldots,\tau_k'$} is a strict Thom system of $H^*_T(M)$. Then, {for any $i$}, by Lemma \ref{lem:strictthomsys} the element $\tau_i'$ restricts to $0$ on all $X_j$ except for $X_i$. The action on $M\backslash X_i$ is again equivariantly formal by Lemma \ref{lem:complementef}. Thus $\tau_i'$ restricts to $0$ in $H_T^*(M\backslash X_i)$ and therefore to the kernel of $H_T^*(X_i)\rightarrow H_T^*(SX_i)$. The latter is the ideal generated by the equivariant Euler class of $SX$, thus the claim follows.
\end{proof}

\begin{lem}\label{lem:ramified}
Let $M$ be an equivariantly formal compact, orientable $T$-manifold such that the map $H^*(M)\rightarrow H^*(X_i)$ is surjective for all components $X_1,\ldots,X_k$ of $M^T$. Then every isotropy codimension $1$ element of the connected orbit type stratification $\chi$ of $M$ is ramified.
\end{lem}

\begin{proof}
Suppose this is not the case. {Every fixed point component $X_i$ is minimal in $\chi$ and hence ramified. Thus an isotropy codimension 1 element containing multiple of the $X_i$ will be ramified as well because it is necessarily minimal with respect to that property. It follows from the assumption that} there is an isotropy codimension $1$ element ${C\in}\chi$ which contains only a single $X_i$. {By Proposition \ref{prop:eqforminherited} the action on $C$ is again equivariantly formal. Since $H^*(M)\rightarrow H^*(X_i)$ factors through $H^*(C)$ we deduce surjectivity of $H^*(C)\rightarrow H^*(X_i)$. Now by Lemma \ref{lem:complementef} the action on $C\backslash X_i$ is equivariantly formal. However it does not have fixed points, which is a contradiction}.
\end{proof}

For the following lemma, recall our convention that elements in a graded space are assumed to be homogeneous.

\begin{lem}\label{lem:productdecomposition}
Let $X$ be a path-connected, trivial $T$-space.
\begin{enumerate}[(i)]
\item Fix $ x\in H_T^*(X)$ as well as coprime elements $f_1,\ldots,f_k\in R$. If there are $x_1,\ldots,x_k\in H_T^*(X)$ with the properties that $x=\prod_{i=1}^k x_i$ and the $x_i-f_i$ are nilpotent, then $x_1,\ldots,x_k$ are unique.

\item Assume $X$ is a smooth manifold and {$V\rightarrow X$} is an effective $T$-equivariant vector bundle. Then $V$ splits as $V=V_1\oplus \ldots\oplus V_l$ for $T$-equivariant vector bundles $V_i\rightarrow X$ with the property that the identity component of every isotropy in $V_i$ is the same codimension $1$ subtorus $S_i\subset T$ and $S_i\neq S_j$ for $i\neq j$. Let $e$, resp.\ $e_i$, denote the equivariant Euler classes of $V$, resp.\ $V_i$ and let {$\alpha_i\in H^2(BT)$} be a nontrivial weight associated to $S_i$ (i.e.\ which vanishes when restricted to $H^2(BS_i)$). Then $e=\prod_{i=1}^k e_i$ and $e_i-a_i\alpha_i^{k_i}$ is nilpotent for some $k_i\in \mathbb{N}$, $a_i\in \mathbb{Q}^\times$.

\item In the setting of $(ii)$, suppose we have a decomposition $e=\prod_{i=1}^l x_i$ such that $x_i-b_i\beta_i^{l_i}$ is nilpotent for some $b_i\in \mathbb{Q}^\times$, $l_i\in \mathbb{N}$ and pairwise linearly independent {$\beta_i\in H^2(BT)$}. Then there are $c_i,d_i\in\mathbb{Q}^\times$, such that after possibly changing the order we have $\alpha_i=c_i\beta_i$  and $e_i=d_ix_i$, $i=1,\ldots,l$.
\end{enumerate}

\end{lem}

\begin{proof} Suppose first that we have already shown statement $(i)$ for products of $2$ elements. As the action is trivial, $H_T^*(X)\cong R\otimes H^*(X)$ inherits a multiplicative bigrading. By Lemma \ref{lem:trivialaction'}, $x_i-f_i$ being nilpotent is equivalent to the fact that $x_i$ restricts to $f_i$ in $H_T^*(*)=R$ for any point, i.e.\ the $R\otimes H^0(X)$ component of $x_i$ is $f_i$. In particular $\prod_{i\geq l} x_i-\prod_{i\geq l} f_i$ is nilpotent for any $l$ and $(i)$ follows by inductively applying the result for two factors of the form $x_l$ and $\prod_{i\geq l+1} x_i$.

For the proof with $2$ factors, suppose we have coprime elements $f,g\in R$ and $a,b,a',b'\in H_T^*(X)$ such that $ab=a'b'$ and $a-f$, $b-g$, $a'-f$, $b'-g$ are nilpotent. We show inductively that $a\equiv a'$ and $b\equiv b'\mod R\otimes H^{\geq k}(X)$ for all $k$. Starting at $k=1$ we note that an element is nilpotent if and only if it is contained in $R\otimes H^+(X)$. Thus the assumptions imply  $a\equiv f\cdot 1 \equiv a'$ and $b\equiv g\cdot 1\equiv b'\mod R\otimes H^+(X)$. Now suppose $a\equiv a'$, $b\equiv b'\mod R\otimes H^{k-1}(X)$ for some $k\geq 2$. For any element $y\in H_T^*(X)$ write $y_i$ for its component in $R\otimes H^i(X)$ and $y_{\leq i}:=y_0+\cdots+y_i$. Then
$(a_{\leq k}\cdot b_{\leq k})_k=(ab)_k=x_k=(a'b')_k=(a_{\leq k}'\cdot b_{\leq k}')_k$
and thus by induction \[a_k\cdot g+f\cdot b_k=a_k'\cdot g+f\cdot b_k'.\]
Now choose a basis $h_\alpha$ of $H^k(X)$. We write $a_k=\sum u_\alpha h_\alpha$, $b_k=\sum v_\alpha h_\alpha$, $a_k'=\sum u_\alpha' h_\alpha$, and $b_k'=\sum v_\alpha' h_\alpha$ for unique $u_\alpha,v_\alpha,u_\alpha',v_\alpha'\in R$. In particular we obtain the equations $u_\alpha g+fv_\alpha=u_\alpha'g+fv_\alpha'$ and thus $(u_\alpha-u_\alpha')g=f(v_\alpha-v_\alpha')$ in $R$. As $f$ and $g$ are coprime it follows that $f|(u_\alpha-u_\alpha')$. However note that, as $f\otimes 1$ is the $R\otimes H^0(X)$-part of the homogeneous element $a$, the total degree satisfies $\deg(u_\alpha-u_\alpha')=\deg a - k = \deg f-k<\deg f$. Thus $u_\alpha-u_\alpha'=0$ and also $v_\alpha-v_{\alpha}'=0$ which finishes the proof of $(i)$.

For any $U\subset T$, the fixed point set $V^U$ is a subbundle of $V$. {Note that it is a bundle over all of $X$ since the action on $X$ is assumed to be trivial.} This yields the {desired} decomposition $V=V_1\oplus\ldots\oplus V_l$ (see also {\cite[Proposition 1.6.2]{A}}). The product decomposition of $e$ follows from the sum decomposition of $V$. Restricting $V_i$ to a point gives a representation which decomposes into $2$-dimensional representations associated to a weight which is some nonzero rational multiple of $\alpha_i$. Hence the restriction of $e_i$ to $H_T^*(*)\cong R$ is equal to $c_i\alpha^{k_i}$ for some $c_i\in\mathbb{Q}^\times$, $k_i\in \mathbb{N}$. Then the rest of $(ii)$ follows by Lemma \ref{lem:trivialaction'}.

To prove $(iii)$ note that for a product decomposition of $e$ as {described in the statement of the lemma}, the $R\otimes H^0(X)$ component of $e$ is equal to $\prod b_i\beta_i^{l_i}$. Thus $\alpha_i^{k_i}=\beta_i^{l_i}$ up to scalars and order. Since the $\alpha_i^{k_i}$ are coprime, the claim now follows from $(i)$.
\end{proof}

\begin{proof}[Proof of Theorem \ref{thm:encodescohomology}]
By Corollary \ref{cor:stratifequivformal} and Lemma \ref{lem:ramified}, $\chi$ is encoded in $H_T^*(M)$. {To prove the theorem it remains to show that $H_T^*(X)$ encodes the equivariant cohomology of any element of $\chi$ as well as the maps between these equivariant cohomology rings induced by inclusions in $\chi$. This is done as follows.}

As in Lemma \ref{lem:minimalthomsystem}, we algebraically detect the actual equivariant Thom classes $\tau_1,\ldots, \tau_k$ (up to $\mathbb{Q}^\times)$) of the components $X_1,\ldots,X_k$ {of $M^T$} as a minimal strict Thom system. Since the restriction $H_T^*(M)\rightarrow H^*(M)$ is surjective by Lemma \ref{lem:equivformal} {and Definition \ref{defn:eqformal}}, the surjectivity assumption implies that $H_T^*(M)\rightarrow H_T^*(X_i)$ is surjective for all $i$. The kernel of this map is $K_i=\{x\in H_T^*(M)\mid x\tau_i=0\}$. Thus the cohomology of {the} $X_i$ is encoded as $H_T^*(X_i)\cong H_T^*(M)/K_i$. Now let $N\in \chi$ be a submanifold which without loss of generality contains exactly the components $X_1,\ldots,X_l$ and whose principal isotropy has a subtorus $S\subset T$ as identity component. We will compute $H_T^*(N)$ from this data.

Consider a small equivariant tube $DN\subset M$ around $N$ with equivariant sphere bundle $SN\subset DN$. We have a Mayer-Vietoris sequence
\[0\rightarrow H_T^*(X)\rightarrow H_T^*(M-N)\oplus H_T^*(N)\rightarrow H_T^*(SN)\rightarrow 0.\]
Let $J_N\subset H_T^*(N)$ denote the ideal generated by the equivariant Euler class $e_N$ of the normal bundle of $N\subset M$. By Lemma \ref{lem:spherebundle}, $J_N$ is the kernel of $H_T^*(N)\rightarrow H_T^*(SN)$. Consequently, for {$y\in J_N$}, elements of the form $(0,y)$ in the middle term of the Mayer-Vietoris sequence lie in the image of the restriction map on $H_T^*(X)$. This proves that $J_N$ is contained in the image of the restriction $r_N\colon H_T^*(M)\rightarrow H_T^*(N)$. 

We now try to characterize $r_N^{-1}(J_N)$ algebraically. Let $r_i\colon H_T^*(M)\rightarrow H_T^*(M)/K_i\cong H_T^*(X_i)$, $i=1,\ldots,l$, denote the restriction map. Recall that the class $\tau_i$ restricts in $H_T^*(X_i)$ to the equivariant euler class of the normal bundle of $X_i$ in $M$, up to scalar. Then it follows from part $(ii)$ of Lemma \ref{lem:productdecomposition} that we find a product decomposition \[r_i(\tau_i)=\prod_{j=1}^{n_j}x_{ij}\] in which $x_{ij}$ has the property that for some $c_{ij}\in\mathbb{Q}^\times$ and $l_{ij}\in \mathbb{N}$ the element $x_{ij}-c_i\alpha_{ij}^{l_{ij}}$ is nilpotent where $\alpha_{ij}\in H^2(BT)$ are the distinct weights associated to those codimension $1$ isotropy groups which occur around $X_i$. Then by part $(iii)$ of Lemma \ref{lem:productdecomposition}, the $x_{ij}$ are up to scalars the equivariant Euler classes of certain subbundles $V_{ij}$ of the normal bundle $V_i$ of $X_i\subset M$ as in Lemma \ref{lem:productdecomposition} part $(ii)$. The restriction of the normal bundle of $N\subset M$ to $X_i$ is equal to the sum over those $V_i$ whose unique connected maximal isotropy group does not contain the principal isotropy $S$ of $N$. Let $x_i$ be the product over those $x_{ij}$ for which $\alpha_{ij}$ does not vanish on $S$. Then $x_i\in H_T^*(X_i)$ is -- up to scalar -- the restriction of $e_N$. After possibly renormalizing the individual $x_i$ we find an element $\tau_N\in H_T^*(M)$ for which $r_i(\tau_N)=x_i$, for all $i=1,\ldots,l$.
By Lemma \ref{lem:divisiblebye}, $r_N(\tau_N)$ agrees with $e_N$ up to scalar and we have
\[
r_N^{-1}(J_N)=\{x\in H_T^*(M)\mid r_i(\tau_N)|r_i(x)\text{ for }i=1,\ldots,l\}=:I_N.
\]
This description encodes $r_N^{-1}(J_N)$ algebraically, as it is independent of the particular choice of $\tau_N$. As the ideal $I_N$ restricts onto $J_N=e_N\cdot H_T^*(N)$, the map
\[I_N\rightarrow \bigoplus_{i=1}^l H_T^*(X_i)\cong H_T^*(N^T),\quad x\mapsto \left(\frac{r_1(x)}{r_1(\tau_N)},\ldots,\frac{r_l(x)}{r_l(\tau_N)}\right)\]
is well defined and its image is that of the injective map $H_T^*(N)\rightarrow H_T^*(N^T)$. Thus we have constructed $H_T^*(N)$ out of $H_T^*(M)$. If furthermore $N'\subset N$ is another isotropy manifold containing without loss of generality the fixed point components $X_1,\ldots, X_{l'}$, $l'\leq l$, then there is a commutative diagram
\[\xymatrix{H_T^*(N)\ar[d]\ar[r] & \prod_{i=1}^l H_T^*(X_i)\ar[d]\\ H_T^*(N')\ar[r]& \prod_{i=1}^{l'} H_T^*(X_i)}\]
where the horizontal maps are injective and the right hand map is projection onto the first $l'$ factors. We have just argued that $H_T^*(M)$ encodes the image of the horizontal maps, hence it also encodes the left hand map.
\end{proof}

\section{Remarks on the integral case}\label{sec:integral}
This section has the purpose of commenting on the question what additional information on the orbit type stratification can be deduced from the integral equivariant cohomology. Generalizing from our results in the rational case it seems natural to ask:
\begin{enumerate}[(i)]
\item What does $H_T^*(X;\mathbb{Z})$ know about the full orbit type stratification also considering disconnected isotropies?
\item Does $H_T^*(X;\mathbb{Z})$ encode the equivariant integral cohomologies of the strata in an equivariantly formal setting analogous to Theorem \ref{thm:encodescohomology}?
\end{enumerate}
There are some subtleties regarding the right requirements and the notion of equivariant formality in the integral setting. E.g.\ unlike in the rational case, a module of the form $H^*(BT;\mathbb{Z})\otimes H^*(X;\mathbb{Z})$ is not necessarily free over $H^*(BT;\mathbb{Z})$ and in particular freeness is not necessarily implied by the degeneracy of the Serre spectral sequence of the Borel fibration. This problem however does not arise in case $H^*(X;\mathbb{Z})$ is free over $\mathbb{Z}$. Let us begin by pointing out the limitations of possible generalizations even under the assumption of torsion freeness.

\begin{ex}
Consider $T^2$-actions on $S^4\subset \mathbb{C}^2\oplus \mathbb{R}$ given by $(s,t)\cdot (v,w,h)=(s^at^bv,s^c t^d w,h)$, for $a,b,c,d\in\mathbb{Z}$. Then $R=\mathbb{Z}[X,Y]$ and $H_T^*(S^4)$ is {freely} generated over $R$ by $1\in H^0_T(S^4)$, $\alpha\in H^4_T(S^4)$ with a single relation $\alpha^2=(aX+bY)(cX+dY)\alpha$. To see this, note that the action above is a pullback of the standard $T^2$-action on $S^4$, i.e.\ with $(a,b,c,d)=(1,0,0,1)$ {along the homomorphism $\varphi\colon T^2\rightarrow T^2$, $(s,t)\mapsto (s^at^b,s^ct^d)$. The standard action has} the relation $\alpha^2=XY\alpha$ (this follows e.g.\ from the integral GKM description). There is a map induced by ${\varphi}$ between the two equivariant cohomologies, which maps $H(BT;\mathbb{Z})$ bases to one another and transforms $H(BT;\mathbb{Z})$ via $X\mapsto aX+bY$ and $Y\mapsto cX+dY$. Hence the relation for $\alpha^2$ transforms accordingly as claimed above.

\begin{enumerate}[(i)]
\item
Setting $(a,b,c,d)=(2,0,3,0)$ we obtain 5 different elements in the orbit type stratification: two fixed points, $2$ two-spheres with isotropies $\mathbb{Z}_2\times S^1$ and $\mathbb{Z}_3\times S^1$ as well as the principal orbit type $\{1\}\times S^1$. For $(6,0,1,0)$ there are only 4 as only one two-sphere occurs with isotropy $\mathbb{Z}_6\times S^1$. Thus we see that the combinatorics of the orbit type stratification are not encoded even though all extensions are ramified. The connected orbit type stratification is of course encoded by the previous rational results.

\item
Consider the actions with $(a,b,c,d)$ equal to $(2,0,0,3)$ and $(6,0,0,1)$. Then, while the connected orbit type stratification is encoded in $H_T^*(S^4)$, the corresponding equivariant cohomology algebras are not: in both cases $(S^4)^{S^1\times \{1\}}$ is $S^2$ however the rotation speeds of the respective $T$-actions are different which yields nonisomorphic equivariant cohomology algebras. Thus without further restrictions on the combinatorics of the stratification no integral result analogous to Theorem \ref{thm:encodescohomology} can be expected to hold.
\end{enumerate}
\end{ex}

The reason for the failure of these methods lies in the fact that for general subgroups of $S\subset T$ the Borel localization theorem does no longer establish a bijection between Thom systems of $H_S^*(X)$ and components of $X^S$. This is due to fact that the ring $H^*(BS;\mathbb{Z})$ might contain elements of positive degree which multiply to $0$ and thus multiplicatively closed subsets available for localization are somewhat limited.

\begin{ex}
Consider the above example for $(a,b,c,d)=(2,0,0,3)$ and the subgroup $S=S^1\times \mathbb{Z}_2$. One has $H^*(BS)=\mathbb{Z}[X,Y]/(2Y)$ and $H_S^*(S^4)=H^*(BS)\otimes_R H_T^*(S^4)\cong \mathbb{Z}[X,Y,\alpha]/(2Y,\alpha^2)$. The fixed point set $(S^4)^S$ consists of two discrete points. However there is no element in $H_S^*(S^4)$ which restricts to a nontrivial element on a single fixed point while vanishing on the other: $\alpha$ has to restrict to $0$ on the fixed points due to being nilpotent while elements in the image of $H^*(BS)\rightarrow H_S^*(S^4)$ restrict to the same element on both fixed points. Thus the technique of using Thom systems to detect fixed point components does not apply for the subgroup $S$.
\end{ex}

Despite these counterexamples, the integral cohomology $H_T^*(X;\mathbb{Z})$ does of course know more about the orbit type stratification than $H_T^*(X;\mathbb{Q})$. The correspondence between certain Thom systems and fixed point components of $S$ does carry over in case $H^+(BS;\mathbb{Z})-\{0\}$ is multiplicatively closed, enabling Borel localization. The groups $S$ for which this is the case are precisely tori and $p$-tori, i.e.\ subgroups of the form $(\mathbb{Z}_p)^r$, where $p$ is prime. One way to go would therefore be to develop results analogous to those in this paper for $p$-tori (in fact, the references \cite{Q} and \cite[Section 3.6]{AP} mentioned in the introduction and in Remark \ref{rem:somewhatdually} also deal with $p$-torus actions), and deduce refined results on the orbit type stratification from $H_T^*(X;\mathbb{Z})$ via this route.

\end{document}